\documentclass{gCOV2e}

\usepackage{hyperref}
\usepackage{graphicx}

\PassOptionsToPackage{pdfauthor={Vladimir V. Kisil},%
    pdftitle={Poincare Extension of Moebius Transformations},%
    pdfsubject={mathematics},%
    backref=page,%
  pdfkeywords={symmetries}
}{hyperref}
\IfFileExists{eulervm.sty}{%
\usepackage{eulervm}
\newcommand{\Aprime}{A\!'}}
{\newcommand{\Aprime}{A'}
\let\mathbold=\mathbf}
\providecommand{\cites}[1]{\cite{#1}}
\providecommand{\amscite}[3]{\cite[#3]{#1}}
\providecommand{\citelist}[1]{#1}

\DeclareFontFamily{OT1}{cyr}{}
\DeclareFontShape{OT1}{cyr}{m}{n}
   {  <5> <6> <7> <8> <9> gen * wncyr
      <10> <10.95> <12> <14.4> <17.28> <20.74> <24.88> wncyr10}{}
\DeclareFontShape{OT1}{cyr}{m}{it}
    {
       <5> <6> <7> <8> <9> gen * wncyi
      <10> <10.95> <12> <14.4> <17.28> <20.74> <24.88>wncyi10
      }{}
\DeclareFontShape{OT1}{cyr}{m}{ss}
    {
       <5> <6> <7> <8> wncyss8
       <9> wncy9
      <10> <10.95> <12> <14.4> <17.28> <20.74> <24.88>wncyss10
      }{}
\DeclareFontShape{OT1}{cyr}{m}{sc}
    {
       <5> <6> <7> <8> <9> <10> <10.95> <12> <14.4> <17.28> <20.74> <24.88>wncysc10
      }{}
\DeclareFontShape{OT1}{cyr}{bx}{n}
   {
       <5> <6> <7> <8> <9> gen * wncyb
      <10> <10.95> <12> <14.4> <17.28> <20.74> <24.88>wncyb10
      }{}
\DeclareTextFontCommand{\textcyr}{\fontfamily{cyr}\selectfont}
\providecommand{\cyr}{\fontfamily{cyr}\selectfont\def\cprime{\~}}
\providecommand{\cprime}{'}

\providecommand{\Space}[3][]{\ifx#2R\ifx#1e \mathbb{C}^{#3} \else
\ifx#1p \mathbb{D}^{#3} \else
\ifx#1h \mathbb{O}^{#3} \else
\ifx#1\sigma \mathbb{A}^{#3} \else
\ensuremath{\mathbb{#2}^{#3}_{#1}{}} \fi \fi \fi \fi \else
\ensuremath{\mathbb{#2}^{#3}_{#1}{}} \fi}
\providecommand{\dSpace}[2]{\dot{\mathbb{#1}}^{#2}}
\providecommand{\Space}[3][]{\ensuremath{\mathbb{#2}^{#3}_{#1}{}}}
\providecommand{\SL}[1][2]{\ensuremath{\mathrm{SL}_{#1}(\Space{R}{})}}
\providecommand{\oper}[1]{\mathcal{#1}}
\providecommand{\such}{\,\mid\,}
\providecommand{\tr}{\mathop{\mathrm{tr}}}
\providecommand{\scalar}[3][\relax]{\left\langle #2,#3 
        \right\rangle\ifx#1\relax\else_{#1}\fi}
\providecommand{\modulus}[2][\relax]{\left| #2 \right|\ifx#1\relax\else_{#1}\fi}
  \providecommand{\Zbl}[1]{Zbl\href{http://www.emis.de:80/cgi-bin/zmen/ZMATH/en/zmathf.html?first=1&maxdocs=3&type=html&an=#1&format=complete}{#1}}

\providecommand{\cycle}[3][]{{#1 C^{#2}_{#3}}}

\providecommand{\bs}{\tau}

\providecommand{\map}[1]{\mathbold{#1}}
\providecommand{\alli}{\iota}

\newtheorem{thm}{Theorem}
\newtheorem{prop}[thm]{Proposition}
\newtheorem{lem}[thm]{Lemma}
\newtheorem{cor}[thm]{Corollary}

\theoremstyle{definition}
\newtheorem{defn}[thm]{Definition}
\theoremstyle{remark}
\newtheorem{rem}[thm]{Remark}

\usepackage{textcomp}

\begin{document}
\title{Poincar\'e Extension of M\"obius Transformations}

\author
{\name{\href{http://www.maths.leeds.ac.uk/~kisilv/}{Vladimir V. Kisil}}
\thanks{On  leave from Odessa University.}
\thanks{Email: \texttt{kisilv@maths.leeds.ac.uk}.
Web: {http://www.maths.leeds.ac.uk/\texttildelow{}kisilv/}}
\affil
{School of Mathematics,
University of Leeds,
Leeds LS2\,9JT,
England} 
\thanks{AMSMSC[2010]: Primary 30C20; Secondary 30C35.}
}
\maketitle

\begin{abstract}
  Given sphere preserving (M\"obius) transformations in
  \(n\)-dimensional
  Euclidean space one can use the Poincar\'e extension to obtain
  sphere preserving transformations in a half-space of \(n+1\)
  dimensions. The Poincar\'e extension is usually provided either by
  an explicit formula or by some geometric construction. We
  investigate its algebraic background and describe all available
  options. The solution is given in terms of one-parameter subgroups
  of M\"obius transformations acting on triples of quadratic forms.
  To focus on the concepts, this paper deals with the M\"obius
  transformations of the real line only.
\end{abstract}



\centerline{\emph{Dedicated to John Ryan on the occasion of his 60th birthday}}

\section{Introduction}
\label{sec:introduction}

It is known, that M\"obius transformations on \(\Space{R}{n}\) can be
expanded to the ``upper'' half-space in \(\Space{R}{n+1}\) using the
Poincar\'e extension~\citelist{\amscite{Beardon95}*{\S~3.3}
  \amscite{JParker07a}*{\S~5.2}}.  An explicit formula is usually
presented without a discussion of its origin. In particular, one may
get an impression that the solution is unique. This paper considers
various aspects of such extension and describes different possible
realisations. Our consideration is restricted to the case of extension
from the real line to the upper half-plane. However, we made an effort
to present it in a way, which allows numerous further generalisations.

\section{Geometric construction}
\label{sec:geom-constr}

We start from the geometric procedure in the standard situation. The
group \(\SL\) consists of  real \(2\times 2\) matrices with the unit determinant.
\(\SL\) acts on the real line by linear-fractional maps:
\begin{equation}
  \label{eq:moebius-map-defn}
  \begin{pmatrix}
    a&b\\c&d
  \end{pmatrix}: x \mapsto \frac{ax+b}{cx+d},
  \quad \text{ where } 
  x\in\Space{R}{}\text{ and } 
    \begin{pmatrix}
    a&b\\c&d
  \end{pmatrix}\in\SL.
\end{equation}
A pair of (possibly equal) real numbers \(x\)
and \(y\)
uniquely determines a semicircle \(\cycle{}{xy}\)
in the upper half-plane with the diameter \([x,y]\).
For a linear-fractional transformation
\(M\)~\eqref{eq:moebius-map-defn},
the images \(M(x)\)
and \(M(y)\)
define the semicircle with the diameter \([M(x),M(y)]\),
thus, we can define the action of \(M\)
on semicircles: \(M(\cycle{}{xy})=\cycle{}{M(x)M(y)}\).
Geometrically, the Poincar\'e extension is based on the following
lemma, see Fig.~\ref{fig:poincare-extension}(a) and more general
Lem.~\ref{le:poincare-geom-general} below:
\begin{lem}
  \label{le:pencil-common-point}
  If a pencil of semicircles in the upper half-plane has a common
  point, then the images of these semicircles under a
  transformation~\eqref{eq:moebius-map-defn} have a common point as
  well.
\end{lem}
Elementary geometry of right triangles tells that a pair of
intersecting intervals \([x,y]\), \([x',y']\), where \(x<x'<y<y'\),
defines the point
\begin{equation}
  \label{eq:poincare-point-ell}
  \left( \frac{xy-x'y'}{x+y-x'-y'}\,,\,
    \frac{\sqrt{(x-y')(x-x')(x'-y)(y-y')}}{x+y-x'-y'}\right) \in \Space[+]{R}{2}.
\end{equation}
An alternative demonstration uses three observations:
\begin{enumerate}
\item the scaling \(x\mapsto ax\),
  \(a>0\)
  on the real line produces the scaling \((u,v) \mapsto(au,av)\)
  on pairs~\eqref{eq:poincare-point-ell};
\item the horizontal shift \(x\mapsto x+b\) on the real line produces
  the horizontal shift \((u,v) \mapsto (u+b,v)\) on
  pairs~\eqref{eq:poincare-point-ell}; 
\item for the special case \(y=-x^{-1}\) and \(y'=-{x'}^{-1}\) 
  the pair~\eqref{eq:poincare-point-ell} is \((0,1)\). 
\end{enumerate}
Finally, expression~\eqref{eq:poincare-point-ell}, as well
as~\eqref{eq:poincare-point-hyp}--\eqref{eq:poincare-point-par} below,
can be calculated by the specialised CAS for M\"obius invariant
geometries~\cites{Kisil05b,Kisil14b}.

\begin{figure}[htbp]
  \centering
  (a)\,\includegraphics[scale=.9]{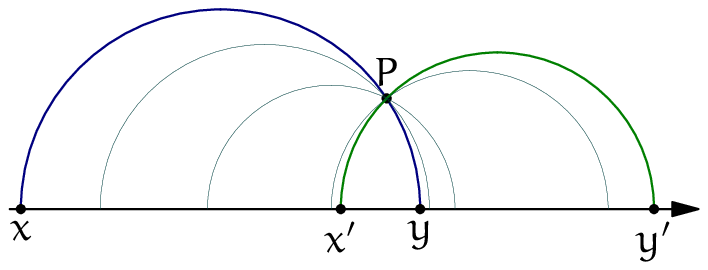}\hfill
  (b)\,\includegraphics[scale=.9]{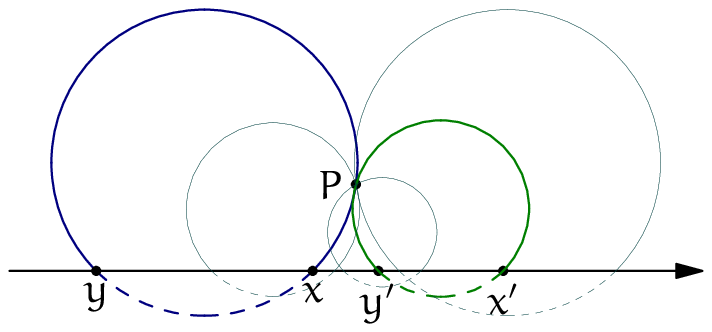}\\[2em]
  (c)\,\includegraphics[scale=.9]{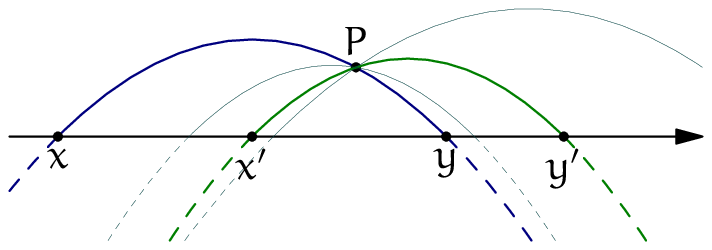}\hfill
  (d)\,\includegraphics[scale=.9]{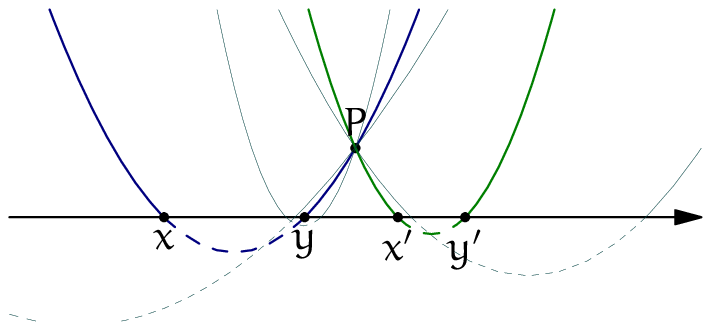}\\[2em]
  (e)\,\includegraphics[scale=.9]{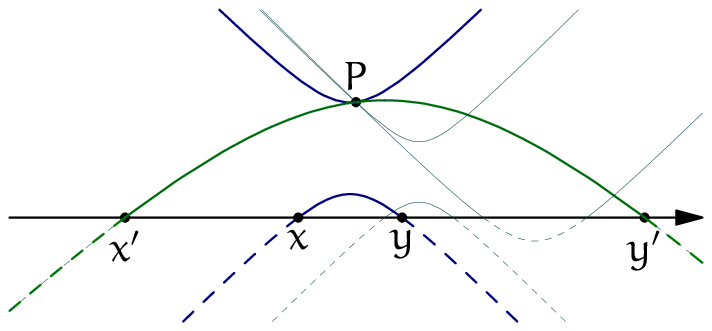}\hfill
  (f)\,\includegraphics[scale=.9]{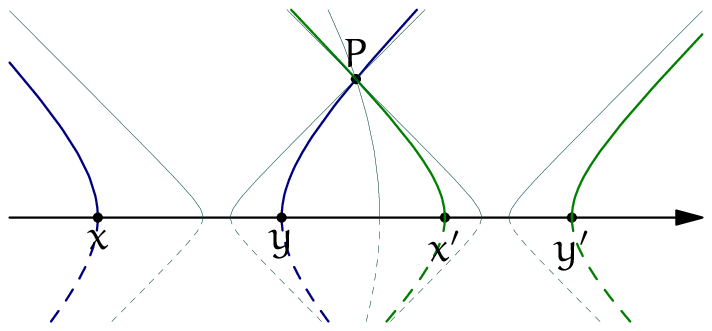}
  \caption[Poincar\'e extensions]{Poincar\'e extensions: first column
    presents points defined by the intersecting intervals \([x,y]\)
    and \([x',y']\), the second column---by disjoint intervals. Each
    row uses the same type of conic sections---circles, parabolas and
    hyperbolas respectively. Pictures are produced by the software~\cite{Kisil14b}.}
  \label{fig:poincare-extension}
\end{figure}

This standard approach can be widened as follows. The above semicircle
can be equivalently described through the unique circle passing \(x\)
and \(y\)
and orthogonal to the real axis. Similarly, an interval \([x,y]\)
uniquely defines a right-angle hyperbola in \(\Space{R}{2}\)
orthogonal to the real line and passing (actually, having her vertices
at) \((x,0)\)
and \((y,0)\).
An intersection with the second such hyperbola having vertices \((x',0)\)
and \((y',0)\)
defines a point with coordinates, see
Fig.~\ref{fig:poincare-extension}(f):
\begin{equation}
  \label{eq:poincare-point-hyp}
  \left( \frac{xy-x'y'}{x+y-x'-y'}\,,\,
    \frac{\sqrt{(x-y')(x-x')(x'-y)(y'-y)}}{x+y-x'-y'}\right), 
\end{equation}
where \(x<y<x'<y'\).  Note, the opposite sign of the product under the
square roots in~\eqref{eq:poincare-point-ell}
and~\eqref{eq:poincare-point-hyp}.

If we wish to consider the third type of conic
sections---parabolas---we cannot use the unaltered procedure: there is
no a non-degenerate parabola orthogonal to the real line and
intersecting the real line in two points. We may recall, that a circle
(or hyperbola) is orthogonal to the real line if its centre belongs to
the real line. Analogously, a parabola is \emph{focally orthogonal} (see
\amscite{Kisil12a}*{\S~6.6} for a general consideration) to the real line
if its focus belongs to the real line. Then, an
interval \([x,y]\) uniquely defines a downward-opened parabola with
the real roots \(x\) and \(y\) and focally orthogonal to the real line.  Two
such parabolas defined by intervals \([x,y]\) and \([x',y']\) have a
common point, see Fig.~\ref{fig:poincare-extension}(c):
\begin{equation}
  \label{eq:poincare-point-par}
  \left(
    \frac{xy'-yx'+D}{x-y-x'+y'}\,,\,
     \frac{(x'-x) (y'-y) (y-x+y'-x')+(x+y-x'-y') D}{(x-y-x'+y')^{2}}
\right),
\end{equation}
where \(D=\pm\sqrt{(x-x')(y-y')(y-x)(y'-x')}\). 
For  pencils of such hyperbolas and parabolas respective variants of
Lem.~\ref{le:pencil-common-point} hold.

Focally orthogonal parabolas make the angle \(45^\circ\) with the real
line. This suggests to replace orthogonal circles and hyperbolas by
conic sections with a fixed angle to the real line, see
Fig.~\ref{fig:poincare-extension}(b)--(e). Of course, to be consistent
this procedure requires a suitable modification of
Lem.~\ref{le:pencil-common-point}, we will obtain it as a byproduct
of our study, see Lem.~\ref{le:poincare-geom-general}. However, the
respective alterations of the above formulae
\eqref{eq:poincare-point-ell}--\eqref{eq:poincare-point-par} become
more complicated in the general case.


The considered geometric construction is elementary and visually
appealing. Now we turn to respective algebraic consideration. 

\section{M\"obius transformations and Cycles}
\label{sec:cycles}
The group \(\SL\) acts on \(\Space{R}{2}\) by matrix multiplication on
column vectors:
\begin{equation}
  \label{eq:left-matrix-mult}
  \oper{L}_g:\ 
  \begin{pmatrix}
    x_1\\x_2
  \end{pmatrix}
  \mapsto
  \begin{pmatrix}
    ax_1+bx_2\\cx_1+dx_2
  \end{pmatrix}=
  \begin{pmatrix}
    a&b\\c&d
  \end{pmatrix}  \begin{pmatrix}
    x_1\\x_2
  \end{pmatrix}
  \,, \quad \text{ where } g=\begin{pmatrix}
    a&b\\c&d
  \end{pmatrix}\in\SL.
\end{equation}
A linear action respects the equivalence relation \(  \begin{pmatrix}
  x_1\\x_2
\end{pmatrix}
\sim   \begin{pmatrix}
  \lambda x_1\\ \lambda x_2
\end{pmatrix}\), \(\lambda\neq 0\) on \(\Space{R}{2}\). The
collection of all cosets for non-zero vectors in \(\Space{R}{2}\) is the \emph{projective
  line} \(P\Space{R}{1}\). Explicitly, a
non-zero vector \(\begin{pmatrix} x_1\\x_2
\end{pmatrix}\in\Space{R}{2}\) corresponds to the point with
homogeneous coordinates
\([x_1:x_2]\in P\Space{R}{1}\). If \(x_2\neq 0\) then this point is
represented by \([\frac{x_1}{x_2}:1]\) as well. 
The embedding \(\Space{R}{} \rightarrow P\Space{R}{1}\) defined by  \(x \mapsto
[x:1]\), 
\(x\in\Space{R}{}\) covers the all but one
of points in \(P\Space{R}{1}\). The exceptional point \([1:0]\) is
naturally identified with the infinity.

The linear action~\eqref{eq:left-matrix-mult} induces a morphism of
the projective line \(P\Space{R}{1}\),
which is called a M\"obius transformation. Considered on the real line
within \(P\Space{R}{1}\),
M\"obius transformations takes fraction-linear form:
\begin{displaymath}
  g:\ 
  [x:1]
  \mapsto
  \left[\frac{ax+b}{cx+d}:1\right]\,, \quad \text{ where } g=\begin{pmatrix}
    a&b\\c&d
  \end{pmatrix}\in\SL \text{ and } cx+d\neq 0.
\end{displaymath}
This \(\SL\)-action on \(P\Space{R}{1}\) is denoted as \(g: x\mapsto
g\cdot x\).  We note that the correspondence of
column vectors and row vectors \( 
\map{i}: \begin{pmatrix}
  x_1\\x_2
\end{pmatrix} \mapsto (x_2, -x_1)\) intertwines
the left multiplication \(\oper{L}_g\)~\eqref{eq:left-matrix-mult} and
the right multiplication \(\oper{R}_{g^{-1}}\) by the inverse matrix:
\begin{equation}
  \label{eq:right-matrix-mult}
  \oper{R}_{g^{-1}}:\ 
    (x_2,-x_1)
  \mapsto
    (cx_1+dx_2,\, -ax_1-bx_2)
    =
    (x_2,-x_1)
  \begin{pmatrix}
    d&-b\\-c&a
  \end{pmatrix}\,.
\end{equation}
We extended the map \(\map{i}\) to \(2\times 2\)-matrices by the rule:
\begin{equation}
  \label{eq:t-map-matrix}
  \map{i}:\
  \begin{pmatrix}
    x_1&y_1\\
    x_2&y_2
  \end{pmatrix}
  \mapsto
  \begin{pmatrix}
    y_2&-y_1\\
    x_2&-x_1
  \end{pmatrix}\,.
\end{equation}
Two columns \(\begin{pmatrix} x\\1
\end{pmatrix}\) and \(\begin{pmatrix} y\\1
\end{pmatrix}\) form the \(2\times 2\) matrix \(M_{xy}=\begin{pmatrix} x&y\\1&1
\end{pmatrix}\). For geometrical reasons appearing in
Cor.~\ref{co:null-set-quadratic}, we call a \emph{cycle} the
\(2\times 2\)-matrix \(\cycle{}{xy}\) defined by
\begin{equation}
  \label{eq:matrix-for-cycle-defn}
\cycle{}{xy}=\frac{1}{2} M_{xy}\cdot\map{i}(M_{xy})=\frac{1}{2} M_{yx}\cdot\map{i}(M_{yx})=
  \begin{pmatrix}
    \frac{x+y}{2}  & -xy\\
    1 & -\frac{x+y}{2}
  \end{pmatrix}\,.
\end{equation}
We note that \(\det \cycle{}{xy}=-(x-y)^2/4\), thus \(\det \cycle{}{xy}=0\) if and
only if \(x=y\). Also, we can consider the M\"obius transformation
produced by the \(2\times 2\)-matrix \(\cycle{}{xy}\) and calculate:
\begin{equation}
  \label{eq:eigenvectors-c_xy}
\cycle{}{xy}
  \begin{pmatrix}
  x\\1  
  \end{pmatrix}=
\lambda   \begin{pmatrix}
  x\\1  
  \end{pmatrix}
\quad \text{and } \quad
\cycle{}{xy}
  \begin{pmatrix}
  y\\1  
  \end{pmatrix}=
-\lambda   \begin{pmatrix}
  y\\1  
  \end{pmatrix}
\quad\text{where } \lambda=\frac{x-y}{2}.
\end{equation}
Thus, points \([x:1]\), \([y:1]\in P\Space{R}{1}\) are fixed by
\(\cycle{}{xy}\). Also,  \(\cycle{}{xy}\) swaps the interval \([x,y]\) and its
complement. 

Due to their structure, matrices \(\cycle{}{xy}\) can be parametrised by
points of \(\Space{R}{3}\).  Furthermore, the map from
\(\Space{R}{2}\rightarrow \Space{R}{3}\) given by \((x,y)\mapsto
\cycle{}{xy}\) naturally induces the projective map \((P\Space{R}{1})^2
\rightarrow P\Space{R}{2}\) due to the identity:
\begin{displaymath}
\frac{1}{2}  \begin{pmatrix}
    \lambda x & \mu y\\
    \lambda & \mu
  \end{pmatrix}
  \begin{pmatrix}
    \mu & -\mu y\\
    \lambda & - \lambda x
  \end{pmatrix}=\lambda \mu    \begin{pmatrix}
    \frac{x+y}{2} & -xy\\
    1 & -\frac{x+y}{2}
  \end{pmatrix}= \lambda \mu \cycle{}{xy}\,.
\end{displaymath}
Conversely, a zero-trace matrix \(
\begin{pmatrix}
  a&b\\c&-a
\end{pmatrix}
\) with a non-positive determinant is projectively equivalent to
a product \(\cycle{}{xy}\)~\eqref{eq:matrix-for-cycle-defn} with
\(x,y=\frac{a\pm\sqrt{a^2+bc}}{c}\). In particular, we can embed a point
 \([x:1]\in P\Space{R}{1}\) to \(2\times 2\)-matrix \(\cycle{}{xx}\) with
 zero determinant. 

The combination
of~\eqref{eq:left-matrix-mult}--\eqref{eq:matrix-for-cycle-defn} implies
that the correspondence \((x,y)\mapsto \cycle{}{xy}\) is \(\SL\)-covariant in the
following sense:
\begin{equation}
  \label{eq:covariance-Cxy}
  g \cycle{}{xy}  g^{-1}=\cycle{}{x'y'}\,,\quad \text{ where }
  x'=g\cdot x \text{ and } y'=g\cdot y.
\end{equation}

To achieve a geometric interpretation of all matrices, we consider the
bilinear form \(Q: \Space{R}{2}\times \Space{R}{2}\rightarrow
\Space{R}{}\) generated by a \(2\times 2\)-matrix \(\begin{pmatrix}
    a&b\\
    c&d
  \end{pmatrix}\):
\begin{equation}
  \label{eq:bilinear-form-defn}
  Q(x, y)=
  \begin{pmatrix}
    x_1 & x_2
  \end{pmatrix}
  \begin{pmatrix}
    a&b\\
    c&d
  \end{pmatrix}
  \begin{pmatrix}
    y_1 \\ y_2
  \end{pmatrix}\, , \quad \text{ where } x=(x_1,x_2),\ y=(y_1,y_2). 
\end{equation}
Due to linearity of \(Q\), the null set 
\begin{equation}
  \label{eq:bilin-form-null-set}
  \{(x,y)\in \Space{R}{2}\times \Space{R}{2} \such Q(x,y)=0\}
\end{equation}
factors to a subset of \(P\Space{R}{1}\times
P\Space{R}{1}\). Furthermore, for the matrices
\(\cycle{}{xy}\)~\eqref{eq:matrix-for-cycle-defn}, a direct calculation
shows that:
\begin{lem}
  \label{lem:biliniar-inner}
  The following identity holds:
  \begin{equation}
    \label{eq:cycle-inner-product}
    \cycle{}{xy}(\map{i}(x'),y')=\tr (\cycle{}{xy} \cycle{}{x'y'})
      =\textstyle  \frac{1}{2}(x+y)(x'+ y')-(x y+x' y')\,.
  \end{equation}
  In particular, the above expression is a symmetric function of the pairs
  \((x,y)\) and \((x',y')\).
\end{lem}
The map \(\map{i}\) appearance in~\eqref{eq:cycle-inner-product} is
justified once more by the following result.
\begin{cor}
  \label{co:null-set-quadratic}
  The null set of the quadratic form
  \(\cycle{}{xy}(x')=\cycle{}{xy}(\map{i}(x'),x')\) consists of two points \(x\)
  and \(y\).
\end{cor}
Alternatively, the identities \(\cycle{}{xy}(x)=\cycle{}{xy}(y)=0\) follows
from~\eqref{eq:eigenvectors-c_xy} and the fact that \(\map{i}(z)\) is
orthogonal to \(z\) for all \(z\in\Space{R}{2}\).  Also, we note that:
\begin{displaymath}
  \map{i}\left(\!\begin{pmatrix}
    x_1 \\ x_2
  \end{pmatrix}\!\right)
  \begin{pmatrix}
    a&b\\
    c&d
  \end{pmatrix}
  \begin{pmatrix}
    y_1 \\ y_2
  \end{pmatrix}=
  \begin{pmatrix}
    x_1 & x_2
  \end{pmatrix}
  \begin{pmatrix}
    -c&-d\\
    a&b
  \end{pmatrix}
  \begin{pmatrix}
    y_1 \\ y_2
  \end{pmatrix}.
\end{displaymath}

Motivated by Lem.~\ref{lem:biliniar-inner}, we call
\(\scalar{\cycle{}{xy}}{\cycle{}{x'y'}}:=-\tr (\cycle{}{xy} \cycle{}{x'y'})\) the
\emph{pairing} of two cycles.  It shall be noted that
the pairing is \emph{not} positively defined, this follows from the
explicit expression~\eqref{eq:cycle-inner-product}. The sign is chosen
in such way, that
\begin{displaymath}
\scalar{\cycle{}{xy}}{\cycle{}{xy}}=-2\det(\cycle{}{xy})=\textstyle \frac{1}{2}(x-y)^2\geq 0.  
\end{displaymath}
Also, an immediate consequence of Lem.~\ref{lem:biliniar-inner} or
identity~\eqref{eq:bilinear-form-defn} is
\begin{cor}
  The  pairing of cycles is invariant under the 
  action~\eqref{eq:covariance-Cxy} of \(\SL\):
  \begin{displaymath}
    \scalar{g\cdot \cycle{}{xy} \cdot g^{-1} }{g\cdot \cycle{}{x'y'} \cdot g^{-1} } =\scalar{\cycle{}{xy}}{\cycle{}{x'y'}}.
  \end{displaymath}
\end{cor}
From~\eqref{eq:cycle-inner-product}, the null
set~\eqref{eq:bilin-form-null-set} of the form \(Q=\cycle{}{xy}\) can be
associated to the family of cycles \(\{\cycle{}{x'y'} \such
\scalar{\cycle{}{xy}}{\cycle{}{x'y'}}=0, (x',y')\in\Space{R}{2}\times
\Space{R}{2}\}\) which we will call \emph{orthogonal} to \(\cycle{}{xy}\).

\section{Extending cycles}
\label{sec:extending-cycles}

Since bilinear forms with matrices \(\cycle{}{xy}\) have numerous geometric
connections with \(P\Space{R}{1}\), we are looking for a similar
interpretation of the generic matrices. The previous discussion
identified the key ingredient of the recipe: \(\SL\)-invariant
pairing~\eqref{eq:cycle-inner-product} of two forms. Keeping in mind
the structure of \(\cycle{}{xy}\), we will parameterise\footnote{Further
  justification of this parametrisation will be obtained from the
  equation of a quadratic curve~\eqref{eq:quadratics-equation}.} a
generic \(2\times 2\) matrix as \(\begin{pmatrix} l+n&-m\\k&-l+n
\end{pmatrix}\) and consider the corresponding four dimensional vector
\((n,l,k,m)\). Then, the similarity with \(\begin{pmatrix}
    a&b\\c&d
  \end{pmatrix}
  \in\SL\):
\begin{equation}
  \label{eq:sl2-similarity-2x2-matrices}
  \begin{pmatrix}
    {l}'+{n}'&-{m}'\\{k}'&-{l}'+{n}'
  \end{pmatrix}
  =
  \begin{pmatrix}
    a&b\\c&d
  \end{pmatrix}
  \begin{pmatrix}
    l+n&-m\\k&-l+n
  \end{pmatrix}
  \begin{pmatrix}
    a&b\\c&d
  \end{pmatrix}^{-1},
\end{equation}
corresponds to the linear transformation of \(\Space{R}{4}\),
cf.~\amscite{Kisil12a}*{Ex.~4.15}: 
\begin{equation}
  \label{eq:SL2-act-cycle-linear}
  \begin{pmatrix}
    {n}'\\{l}'\\{k}'\\{m}'
  \end{pmatrix}=
  \begin{pmatrix}
    1&0&0&0\\
    0&c b+a d& b d&c a\\
    0&2 c d&d^2&c^2\\
    0&2 a b&b^2&a^2
  \end{pmatrix}
  \begin{pmatrix}
    n\\l\\k\\m
  \end{pmatrix}.
\end{equation}
In particular, this action commutes with the scaling of the first component:
\begin{equation}
  \label{eq:scaling-n}
  \lambda: (n,l,k,m) \mapsto  (\lambda n,l,k,m).
\end{equation}
This expression is helpful in proving the following statement.
\begin{lem}
  \label{le:sl2-invariant-pairing}
  Any \(\SL\)-invariant (in the sense of the
  action~\eqref{eq:SL2-act-cycle-linear}) pairing in \(\Space{R}{4}\) is
  isomorphic to
  \begin{displaymath}
  2\bs \tilde{n}n-2\tilde{l}l+\tilde{k}m+\tilde{m}k
  =
  \begin{pmatrix}
    \tilde{n}&\tilde{l}&\tilde{k}&\tilde{m}
  \end{pmatrix}
  \begin{pmatrix}
    2\bs&0&0&0\\0&-2&0&0\\0&0&0&1\\0&0&1&0
  \end{pmatrix}
  \begin{pmatrix}
    n\\l\\k\\m
  \end{pmatrix},
  \end{displaymath}
  where \(\bs =-1\), \(0\) or \(1\) and 
  \((n,l,k,m)\),  \((\tilde{n},\tilde{l},\tilde{k},\tilde{m})\in\Space{R}{4}\) .
\end{lem}
\begin{proof}
  Let \(T\) be \(4\times 4\) a matrix
  from~\eqref{eq:SL2-act-cycle-linear}, if a \(\SL\)-invariant pairing
  is defined by a \(4\times 4\) matrix \(J=(j_{fg})\), then
  \(T'JT=J\), where \(T'\) is transpose of \(T\). The equivalent
  identity \(T'J=JT^{-1}\) produces a system of homogeneous linear
  equations which has the generic solution:
  \begin{align*}
    j_{12}&=
    j_{13}=
    j_{14}=
    j_{21}=
    j_{31}=
    j_{41}=0,\\
    j_{22}&= \frac{ (d-a) j_{42}-2 b j_{44}}{c}-2  j_{43},&
    j_{23}&=- \frac{b}{c} j_{42},&
    j_{24}&=-\frac{ c j_{42}+2 (a-d) j_{44}}{c},\\
    j_{34}&=\frac{c (a - d) j_{42}+(a -d)^2j_{44}}{ c^2}+ j_{43},&
    j_{33}&=  \frac{b^2}{c^2} j_{44} ,&
    j_{32}&= \frac{b}{c}\frac{c  j_{42}+2  (a-d) j_{44}}{c},
  \end{align*}
  with four free variables \(j_{11}\), \(j_{42}\), \(j_{43}\) and
  \(j_{44}\). Since a solution shall not depend on \(a\), \(b\),
  \(c\), \(d\), we have to put \(j_{42}=j_{44}=0\). Then by the
  homogeneity of the identity \(T'J=JT^{-1}\), we can scale \(j_{43}\)
  to \(1\). Thereafter,  an independent (sign-preserving)
  scaling~\eqref{eq:scaling-n} of \(n\) leaves only three
  non-isomorphic values \(-1\), \(0\), \(1\) of \(j_{11}\).
\end{proof}
The appearance of the three essential different cases \(\bs =-1\),
\(0\) or \(1\) in Lem.~\ref{le:sl2-invariant-pairing} is a
manifestation of the common division of mathematical objects into
elliptic, parabolic and hyperbolic cases \citelist{\cite{Kisil06a}
  \amscite{Kisil12a}*{Ch.~1}}.  Thus, we will use letters ``e'', ``p'',
``h'' to encode the corresponding three values of \(\bs\). 

Now we may describe all \(\SL\)-invariant pairings of bilinear forms.
\begin{cor}
  Any \(\SL\)-invariant (in the sense of the
  similarity~\eqref{eq:sl2-similarity-2x2-matrices}) pairing between
  two bilinear forms \(Q=  \begin{pmatrix}
    l+n&-m\\k&-l+n
  \end{pmatrix}  \) and \(\tilde{Q}=  \begin{pmatrix}
    \tilde{l}+\tilde{n}&-\tilde{m}\\\tilde{k}&-\tilde{l}+\tilde{n}
  \end{pmatrix}  \) is isomorphic to:
  \begin{align}
    \label{eq:cycle-product-expl}
    \scalar[\tau]{Q}{\tilde{Q}}&=-\tr ({Q}_\tau \tilde{Q})\\
    & =
    2\tau \tilde{n}n-2\tilde{l}l+\tilde{k}m+\tilde{m}k, \quad\text{ where }
    {Q}_\tau=  \begin{pmatrix}
      {l}-\tau{n}&-{m}\\{k}&-{l}-\tau{n}, \nonumber 
    \end{pmatrix}
  \end{align}
  and \(\tau=-1\), \(0\) or \(1\).
\end{cor}
Note that we can explicitly write \(Q_\tau\) for \(Q=
\begin{pmatrix}
  a&b\\c&d
\end{pmatrix}\) as follows:
\begin{displaymath}
  Q_e=\begin{pmatrix}
    a&b\\c&d
  \end{pmatrix},\quad
  Q_p=\begin{pmatrix}
    \frac{1}{2}(a-d)&b\\c&-\frac{1}{2}(a-d)
  \end{pmatrix},\quad
  Q_h=\begin{pmatrix}
    -d&b\\c&-a
  \end{pmatrix}.
\end{displaymath}
In particular, \(Q_h=-Q^{-1}\) and
\(Q_p=\frac{1}{2}(Q_e+Q_h)\). Furthermore, \(Q_p\) has the same
structure as \(\cycle{}{xy}\).  Now, we are ready to extend the
projective line \(P\Space{R}{1}\) to two dimensions using the analogy
with properties of cycles \(\cycle{}{xy}\).
\begin{defn}
  \label{de:product-cycles}
  \begin{enumerate}
  \item Two bilinear forms \(Q\) and \(\tilde{Q}\) are
    \(\tau\)-\emph{orthogonal} if \(\scalar[\tau]{Q}{\tilde{Q}}=0\). 
  \item \label{it:isotropic}
    A form is \(\tau\)-\emph{isotropic} if it is
    \(\tau\)-orthogonal to itself.
  \end{enumerate}
\end{defn}
If a form \( Q= \begin{pmatrix}
    l+n&-m\\k&-l+n
  \end{pmatrix}
\) has \(k\neq 0\) then we can scale it  to obtain \(k=1\), this form
of \(Q\)
is called \emph{normalised}. A normalised \(\tau\)-isotropic form is
completely determined by its diagonal values: \(\begin{pmatrix}
  u+v&-u^2+\tau v^2\\1&-u+v
\end{pmatrix}\). Thus, the set of such forms is in a one-to-one
correspondence with points of \(\Space{R}{2}\). Finally, a form \(
Q= \begin{pmatrix} l+n&-m\\k&-l+n
  \end{pmatrix}\) is e-orthogonal to the \(\tau\)-isotropic form \(\begin{pmatrix}
  u+v&-u^2+\tau v^2\\1&-u+v
\end{pmatrix}\) if:
\begin{equation}
  \label{eq:quadratics-equation}
  k(u^2-\tau v^2)-2lu-2nv+m=0,
\end{equation}
that is the point \((u,v)\in\Space{R}{2}\) belongs to the quadratic curve
with coefficients \((k,l,n,m)\).

\section{Homogeneous spaces of cycles}
\label{sec:homog-spac-cycl}
Obviously, the group \(\SL\) acts on \(P\Space{R}{1}\) transitively,
in fact it is even \(3\)-transitive in the following sense. We say
that a triple \(\{x_1,x_2,x_3\}\subset P\Space{R}{1}\) of distinct
points is \emph{positively oriented} if
\begin{displaymath}
  \text{ either }\quad x_1<x_2<x_3, \quad\text{ or }\quad  x_3<x_1< x_2,
\end{displaymath}
where we agree that the ideal point \(\infty\in P\Space{R}{1}\) is
greater than  any \(x\in\Space{R}{}\). Equivalently, 
a triple \(\{x_1,x_2,x_3\}\) of reals is positively oriented if:
\begin{displaymath}
  (x_1-x_2)(x_2-x_3)(x_3-x_1)>0.
\end{displaymath}
Also, a triple of distinct points, which is not positively oriented, is
\emph{negatively oriented}. 
A simple calculation based on the resolvent-type identity:
\begin{displaymath}
  \frac{ax+b}{cx+d}-\frac{ay+b}{cy+d}=\frac{(x-y)(ad-bc)}{(cx+b)(cy+d)}
\end{displaymath}
shows that both the positive and negative orientations of triples are
\(\SL\)-invariant. On the other hand, the reflection \(x\mapsto -x\)
swaps orientations of triples. Note, that the reflection is a Moebius
transformation associated to the cycle
\begin{equation}
  \label{eq:matrix-J-defn}
 \cycle{}{0\infty}=\begin{pmatrix}
  1&0\\0&-1
\end{pmatrix}, \quad \text{ with } \det \cycle{}{0\infty}=-1.
\end{equation}

A significant amount of information about M\"obius transformations
follows from the fact, that any continuous one-parametric subgroup of
\(\SL\) is conjugated to one of the three following
subgroups\footnote{A reader may know that \(A\), \(N\) and \(K\) are
  factors in the Iwasawa decomposition \(\SL=ANK\)
  (cf. Cor.~\ref{co:Iwasawa-decomp}, however this important result
  does not play any r\^ole in our consideration.}:
\begin{equation}
  \label{eq:ank-subgroup}
  A=\left\{\begin{pmatrix} e^{-t} & 0\\0&e^{t}
    \end{pmatrix}\right\},\quad
  N=\left\{\begin{pmatrix} 1&t \\0&1
    \end{pmatrix}\right\},\quad
  K=\left\{ \begin{pmatrix}
      \cos t &  -\sin t\\
      \sin t & \cos t
    \end{pmatrix}\right\},
\end{equation}
where \(t\in\Space{R}{}\). Also, it is useful to introduce
subgroups \(\Aprime\) and \(N'\) conjugated to \(A\) and \(N\)
respectively:
\begin{equation}
  \label{eq:a1n1-subgroup}
  \Aprime=\left\{\begin{pmatrix} \cosh t & \sinh t\\ \sinh t & \cosh t 
    \end{pmatrix} \such t\in \Space{R}{}\right\},\qquad
  N'=\left\{\begin{pmatrix} 1&0 \\t&1
    \end{pmatrix} \such t\in \Space{R}{}\right\}.
\end{equation}
Thereafter, all three one-parameter subgroups \(\Aprime\), \(N'\)
and \(K\) consist of all matrices with the universal structure
\begin{equation}
  \label{eq:a1n1k-universal}
\begin{pmatrix}
  a&\tau b\\b &a
\end{pmatrix}
\quad  \text{where }\tau =1,\, 0,\, -1 \text{ for }\Aprime,\, N'\,
\text{ and }
K \text{ respectively}.
\end{equation}
We use the notation \(H_\tau\) for these subgroups. Again, any
continuous one-di\-men\-sio\-nal
subgroup of \(\SL\) is conjugated to
\(H_\tau\) for an appropriate \(\tau\).

We note, that matrices from \(A\), \(N\) and \(K\) with \(t\neq 0\)
have two, one and none different real eigenvalues
respectively. Eigenvectors in \(\Space{R}{2}\) correspond to fixed
points of M\"obius transformations on \(P\Space{R}{1}\). Clearly, the
number of eigenvectors (and thus fixed points) is limited by the
dimensionality of the space, that is two. For this reason, if \(g_1\)
and \(g_2\) take equal values on three different points of
\(P\Space{R}{1}\), then \(g_1=g_2\).

Also, eigenvectors provide an effective classification tool:
\(g\in\SL\) belongs to a one-dimensional continuous subgroup conjugated
to \(A\), \(N\) or \(K\) if and only if the characteristic polynomial
\(\det(g-\lambda I)\) has two, one and none different real root(s)
respectively. We will illustrate an application of fixed points
techniques through the following well-known result, which will be used
later.
\begin{lem}
  \label{le:three-transitive}
  Let \(\{x_1,x_2,x_3\}\) and \(\{y_1,y_2,y_3\}\) be positively
  oriented triples of points in \(\dSpace{R}{}\). Then, there is a
  unique (computable!) M\"obius map \(\phi\in\SL\) with \(\phi(x_j)=y_j\) for
  \(j=1\), \(2\), \(3\).
\end{lem}
\begin{proof}
  Often, the statement is quickly demonstrated through an explicit
  expression for \(\phi\), cf.~\amscite{Beardon05a}*{Thm.~13.2.1}. We
  will use properties of the subgroups \(A\), \(N\) and \(K\) to
  describe an algorithm to find such a map.  First, we notice that it
  is sufficient to show the Lemma for the particular case \(y_1=0\),
  \(y_2=1\), \(y_3=\infty\). The general case can be obtained from
  composition of two such maps. Another useful observation is that the
  fixed point for \(N\), that is \(\infty\), is also a fixed point of
  \(A\).

  Now, we will use subgroups \(K\), \(N\) and \(A\) in order of
  increasing number of their fixed points. First, for any \(x_3\) the
  matrix \(g'=\begin{pmatrix}
    \cos t &  \sin t\\
    -\sin t & \cos t
  \end{pmatrix}\in K\)  such that \(\cot t=x_3\)  maps \(x_3\) to
  \(y_3=\infty\). Let \(x_1'=g'x_1\) and \(x_2'=g'x_2\). Then the matrix
  \(g''=\begin{pmatrix}
    1 &  -x_1'\\
    0 & 1
  \end{pmatrix}\in N\), fixes \(\infty=g' x_3\) and sends \(x'_1\) to
  \(y_1=0\). Let \(x''_2=g''x_2'\), from positive orientation of
  triples we have \(0<x''_2<\infty\). Next, the matrix
  \(g'''=\begin{pmatrix}
    a^{-1} &  0\\
    0 & a
  \end{pmatrix}\in A\) with \(a=\sqrt{x''_2}\)  sends \(x''_2\) to \(1\)
  and fixes both
  \(\infty=g''g'x_3\) and \(0=g''g' x_1\). Thus, \(g=g'''g''g'\) makes
  the required transformation \((x_1,x_2,x_3)\mapsto(0,1,\infty)\).
\end{proof}
\begin{cor}
  Let \(\{x_1,x_2,x_3\}\) and \(\{y_1,y_2,y_3\}\) be two triples with the
  opposite orientations. Then, there is a
  unique M\"obius map \(\phi\in\SL\) with \(\phi\circ \cycle{}{0\infty}(x_j)=y_j\) for
  \(j=1\), \(2\), \(3\).
\end{cor}
We will denote by \(\phi_{XY}\) the unique map from
Lem.~\ref{le:three-transitive} defined by triples
\(X=\{x_1,x_2,x_3\}\) and \(Y=\{y_1,y_2,y_3\}\).  

Although we are not going to use it in this paper, we note that the
following important result~\amscite{Lang85}*{\S~III.1} is an immediate
consequence of our \emph{proof} of Lem.~\ref{le:three-transitive}.
\begin{cor}[Iwasawa decomposition]
  \label{co:Iwasawa-decomp}
  Any element of \(g\in\SL\) is a product \(g=g_A g_N g_K\),  where
  \(g_A\),  \(g_N\) and  \(g_K\) belong to subgroups \(A\), \(N\),
  \(K\) respectively and those factors are uniquely defined.
\end{cor}
In particular, we note that it is not a coincidence that the subgroups
appear in the Iwasawa decomposition \(\SL=ANK\) in order of decreasing
number of their fixed points.

\section{Triples of intervals}
\label{sec:triples-intervals}
We change our point of view and instead of two ordered triples of points
consider three ordered pairs, that is---three intervals. For them we
will need the following definition.
\begin{defn}
  \label{de:aligned-triples-intervals}
  We say that a triple of intervals \(\{[x_1,y_1], [x_2,y_2],
  [x_3,y_3]\}\)  is \emph{aligned} if the triples \(X=\{x_1,x_2,x_3\}\)
  and \(Y=\{y_1,y_2,y_3\}\)  of their endpoints have the same orientation. 
\end{defn}
Aligned triples determine certain one-parameter subgroups of M\"obius
transformations as follows:
\begin{prop}
  \label{pr:triple-subgroup}
  Let \(\{[x_1,y_1], [x_2,y_2], [x_3,y_3]\}\) be an aligned triple of
  intervals.
  \begin{enumerate}
  \item If \(\phi_{XY}\) has at most one fixed point, then there is a
    unique (up to a parametrisation) 
    one-parameter semigroup of
    M\"obius map \(\psi(t)\subset\SL\), which maps \([x_1,y_1]\) to
    \([x_2,y_2]\) and \([x_3,y_3]\):
    \begin{displaymath}
      \psi(t_j)(x_1)=x_j,\quad \psi(t_j)(y_1)=y_j,\qquad \text{ for
        some } t_j\in\Space{R}{}  \text{ and } j=2,3.
    \end{displaymath}
  \item Let \(\phi_{XY}\) have two fixed points \(x<y\) and \(\cycle{}{xy}\)
    be the orientation inverting M\"obius transformation with the
    matrix~\eqref{eq:matrix-for-cycle-defn}. For \(j=1\), 
    \(2\), \(3\), we define:
    \begin{align*}
      x'_j&=x_j,& y'_j&=y_j,& x''_j&=\cycle{}{xy}x_j,&
      y''_j&=\cycle{}{xy}y_j&\text{ if } x<x_j<y;\\
      x'_j&=\cycle{}{xy}x_j,& y'_j&=\cycle{}{xy}y_j,& x''_j&=x_j,& y''_j&=y_j,&
      \text{ otherwise}.
    \end{align*}
    Then, there is a 
    one-parameter
    semigroup of M\"obius map \(\psi(t)\subset\SL\), and \(t_2\),
    \(t_3\in\Space{R}{}\) such that:
    \begin{displaymath}
      \psi(t_j)(x'_1)=x'_j,\quad \psi(t_j)(x''_1)=x''_j,\quad
      \psi(t_j)(y'_1)=y'_j,\quad  \psi(t_j)(y''_1)=y''_j,
    \end{displaymath}
    where \(j=2,3\).
  \end{enumerate}
\end{prop}
\begin{proof}
  Consider the one-parameter subgroup of
  \(\psi(t)\subset \SL\) such that \(\phi_{XY}=\psi(1)\). Note, that
  \(\psi(t)\) and \(\phi_{XY}\) have the same fixed points (if any) and no
  point \(x_j\) is fixed since \(x_j\neq y_j\). If the number of fixed
  points is less than \(2\), then \(\psi(t)x_1\), \(t\in\Space{R}{}\)
  produces the entire real line except a possible single fixed
  point. Therefore, there are \(t_2\) and \(t_3\) such that
  \(\psi(t_2)x_1=x_2\) and \(\psi(t_3)x_1=x_3\). Since \(\psi(t)\) and
  \(\phi_{XY}\) commute for all \(t\) we also have:
  \begin{displaymath}
    \psi(t_j)y_1=\psi(t_j)\phi_{XY} x_1=\phi_{XY}\psi(t_j) x_1=\phi_{XY}
    x_j=y_j,\quad \text{ for } j=2,3.
  \end{displaymath}

  If there are two fixed points \(x<y\), then the open interval
  \((x,y)\) is an orbit for the subgroup \(\psi(t)\). Since all
  \(x'_1\), \(x'_2\) and \(x'_3\) belong to this orbit and \(\cycle{}{xy}\)
  commutes with \(\phi_{XY}\) we may repeat the above reasoning for
  the dashed intervals \([x'_j,y'_j]\). Finally, \(x''_j=\cycle{}{xy}x'_j\)
  and \(y''_j=\cycle{}{xy}y'_j\), where \(\cycle{}{xy}\) commutes with \(\phi\)
  and \(\psi(t_j)\), \(j=2\), \(3\).  Uniqueness of the subgroup
  follows from Lemma~\ref{le:uniqueness-orbit}.
\end{proof}
The group \(\SL\) acts transitively on collection of all cycles
\(\cycle{}{xy}\), thus this is a \(\SL\)-homogeneous space. It is easy to
see that the fix-group of the cycle \(\cycle{}{-1,1}\) is
\(\Aprime\)~\eqref{eq:a1n1-subgroup}. Thus the homogeneous space of
cycles is isomorphic to \(\SL/\Aprime\). 
\begin{lem}
  \label{le:uniqueness-orbit}
  Let \(H\) be a one-parameter continuous subgroup of \(\SL\) and
  \(X=\SL/H\) be the corresponding homogeneous space. If two orbits
  of one-parameter continuous subgroups on \(X\) have at least three
  common points then these orbits coincide.
\end{lem}
\begin{proof}
  Since \(H\) is conjugated either to \(\Aprime\), \(N'\) or \(K\),
  the homogeneous space \(X=\SL/H\) is isomorphic to the upper
  half-plane in double, dual or complex
  numbers~\amscite{Kisil12a}*{\S~3.3.4}.  Orbits of one-parameter
  continuous subgroups in \(X\) are conic sections, which are circles,
  parabolas (with vertical axis) or equilateral hyperbolas (with
  vertical axis) for the respective type of geometry. Any two
  different orbits of the same type intersect at most at two points,
  since an analytic solution reduces to a quadratic equation. 
\end{proof}
Alternatively, we can reformulate Prop.~\ref{pr:triple-subgroup} as
follows: three different cycles \(\cycle{}{x_1y_1}\), \(\cycle{}{x_2y_2}\),
\(\cycle{}{x_3y_3}\) define a one-parameter subgroup, which generate either one
orbit or two related orbits passing the three cycles.

We have seen that the number of fixed points is the key
characteristics for the map \(\phi_{XY}\). The next result gives an
explicit expression for it.
\begin{prop} 
\label{pr:type-of-subgroup}
  The map \(\phi_{XY}\) has zero, one or two fixed points
  if the expression
  \begin{equation}
    \label{eq:discriminant}
    \det
    \begin{pmatrix}
      1&x_1y_1&y_1-x_1\\
      1&x_2y_2&y_2-x_2\\
      1&x_3y_3&y_3-x_3
    \end{pmatrix}^2-4\det\begin{pmatrix}
      x_1&1&y_1\\
      x_2&1&y_2\\
      x_3&1&y_3
    \end{pmatrix}
    \cdot\det
    \begin{pmatrix}
      x_1&-x_1y_1&y_1\\
      x_2&-x_2y_2&y_2\\
      x_3&-x_3y_3&y_3
    \end{pmatrix}
  \end{equation}
  is negative, zero or positive respectively.
\end{prop}
\begin{proof}
  If a M\"obius transformation \(\begin{pmatrix}
    a&b\\c&d
  \end{pmatrix}\) maps \(x_1\mapsto y_1\),  \(x_2\mapsto y_2\),
  \(x_3\mapsto y_3\) and  \(s\mapsto s\), then we have a homogeneous linear
  system, cf.~\amscite{Beardon05a}*{Ex.~13.2.4}: 
  \begin{equation}
    \label{eq:three-pairs-fixed-point}
    \begin{pmatrix}
    x_1&1&-x_1y_1&-y_1\\
    x_2&1&-x_2y_2&-y_2\\
    x_3&1&-x_3y_3&-y_3\\
    s& 1 &-s^2&-s
  \end{pmatrix}
  \begin{pmatrix}
    a\\b\\c\\d
  \end{pmatrix}=
  \begin{pmatrix}
    0\\0\\0\\0
  \end{pmatrix}.
  \end{equation}
  A non-zero solution exists if the determinant of the \(4\times 4\)
  matrix is zero. Expanding it over the last row and 
  rearranging terms we obtain the quadratic equation for the fixed point
  \(s\):
  \begin{displaymath}
    s^2 \det
    \begin{pmatrix}
      x_1&1&y_1\\
      x_2&1&y_2\\
      x_3&1&y_3
    \end{pmatrix}
    +s\det
    \begin{pmatrix}
      1&x_1y_1&y_1-x_1\\
      1&x_2y_2&y_2-x_2\\
      1&x_3y_3&y_3-x_3
    \end{pmatrix}
    +\det
    \begin{pmatrix}
      x_1&-x_1y_1&y_1\\
      x_2&-x_2y_2&y_2\\
      x_3&-x_3y_3&y_3
    \end{pmatrix}
    =0.
  \end{displaymath}
  The value~\eqref{eq:discriminant} is the discriminant of this equation.
\end{proof}
\begin{rem}
  It is interesting to note, that the relation \(ax+b-cxy-dy=0\) used
  in~\eqref{eq:three-pairs-fixed-point} 
  can be stated as e-orthogonality of the cycle \(
  \begin{pmatrix}
    a&b\\c&d
  \end{pmatrix}\) and the isotropic bilinear form \(
  \begin{pmatrix}
    x&-xy\\1&-y
  \end{pmatrix}\).
\end{rem}
 If \(y=g_0\cdot x\) for some \(g_0\in
H_\tau \), then for any \(g\in H_\tau\) we also have
\(y_g=g_0\cdot x_g\), where \(x_g=g\cdot x_g\) and \(y_g=g\cdot
y_g\). Thus, we demonstrated the first part of the following result.
\begin{lem}
  Let \(\tau =1\), \(0\) or \(-1\) and a real constant \(t\neq 0\) be such that
    \(1-\tau t^2>0\).
  \begin{enumerate} 
  \item The collections of intervals:
    \begin{equation}
      \label{eq:tau-invariant-cycles}
      I_{\tau,t} =\left\{[x,\textstyle \frac{x+\tau t}{tx+1}]\such x\in \Space{R}{} \right\}
    \end{equation}
    is preserved by the actions of subgroup \(H_\tau\). Any three
    different intervals from \(I_{\tau,t}\) define the subgroup
    \(H_\tau\) in the sense of Prop.~\ref{pr:triple-subgroup}.
  \item 
    All \(H_\tau\)-invariant bilinear forms compose the family \(P_{\tau,t}=\left\{
      \begin{pmatrix}
        a&\tau b\\
        b& a
      \end{pmatrix}
    \right\}\).
  \end{enumerate}
\end{lem}
The family \(P_{\tau,t}\) consists of the eigenvectors of the
\(4\times 4\) matrix from~\eqref{eq:SL2-act-cycle-linear} with
suitably substituted entries. There is (up to a factor) exactly one
\(\tau\)-isotropic form in \(P_{\tau,t}\), namely \(\begin{pmatrix}
  1&\tau \\
  1& 1
\end{pmatrix}\). We denote this form by \(\alli\). It corresponds to
the point \((0,1)\in\Space{R}{2}\) as discussed after
Defn.~\ref{de:product-cycles}. We may say that the subgroup
\(H_\tau\) fixes the point \(\alli\), this will play an important
r\^ole below. 

\section{Geometrisation of cycles}
\label{sec:geom-cycl}

We return to the geometric version of the Poincar\'e extension
considered in Sec.~\ref{sec:geom-constr} in terms of cycles. 
Cycles of the form \(\begin{pmatrix}
  x&-x^2\\1&-x
\end{pmatrix}\) are \(\tau\)-isotropic for any \(\tau\) and are
parametrised by the point \(x\) of the real line. For a fixed
\(\tau\), the collection of all \(\tau\)-isotropic cycles is a larger
set containing the image of the real line from the previous
sentence. Geometrisation of this embedding is described in the
following result.
\begin{lem}
  \begin{enumerate}
    \item The transformation \(x\mapsto \frac{x+\tau t}{tx+1}\) from the
    subgroup \(H_\tau\), which maps \(x\mapsto y\), corresponds to
    the value \(t=\frac{x-y}{xy-\tau}\).  
  \item The unique (up to a factor) bilinear form \(Q\) orthogonal
    to \(\cycle{}{xx}\), \(\cycle{}{yy}\) and \(\alli\) is
    \begin{displaymath}
      Q=\begin{pmatrix}
        \frac{1}{2}(x+y+xy-\tau)&-xy\\
        1&\frac{1}{2}(-x-y+xy-\tau)
      \end{pmatrix}.
    \end{displaymath}
  \item The defined above \(t\) and  \(Q\) are connected by the identity:
    \begin{equation}
      \label{eq:cosine-real-axis}
      \frac{\scalar[\tau]{Q}{\Space{R}{}}}{\sqrt{\modulus{\scalar[\tau]{Q}{Q}}}}
      =\frac{\tau}{\sqrt{\modulus{t^2-\tau}}}.
    \end{equation}
    Here, the real line is represented by the bilinear form \(\Space{R}{}=
    \begin{pmatrix}
     2^{-1/2}&0\\0&2^{-1/2} 
    \end{pmatrix}
    \) normalised such
    that \(\scalar[\tau]{\Space{R}{}}{\Space{R}{}}=\pm1\).
  \item For a cycle \(Q=
  \begin{pmatrix}
    l+n&-m\\k&-l+n
  \end{pmatrix}
  \), the value
  \(\frac{\scalar[e]{Q}{\Space{R}{}}}{\sqrt{\modulus{\scalar[e]{Q}{Q}}}}
  =-\frac{n}{\sqrt{\modulus{l^2+n^2-km}}}\)
  is equal to the cosine of the angle between the curve \(k(u^2+\tau
  v^2)-2lu-2nv+m=0\)~\eqref{eq:quadratics-equation} and the real line,
  cf.~\amscite{Kisil12a}*{Ex.~5.23}.
  \end{enumerate}
\end{lem}
\begin{proof}
  The first statement is verified by a short calculation. A form \(Q=
  \begin{pmatrix}
    l+n&-m\\k&-l+n
  \end{pmatrix}
  \) in the
  second statement may be calculated from the homogeneous system:
  \begin{displaymath}
    \begin{pmatrix}
      0&-2x&x^2&1\\
      0&-2y&y^2&1\\
      -2&0&-\tau&1
    \end{pmatrix}
    \begin{pmatrix}
      n\\l\\k\\m
    \end{pmatrix}=
    \begin{pmatrix}
      0\\0\\0
    \end{pmatrix},
  \end{displaymath}
  which has the rank \(3\) if \(x\neq y\). The third statement
  can be checked by a calculation as well. Finally, the last item is a
  particular case of the more general statement as indicated. Yet, we
  can derive it here from the implicit derivative
  \(\frac{dv}{du}=\frac{ku-l}{n}\) of the function \(k(u^2+\tau
  v^2)-2lu-2nv+m=0\)~\eqref{eq:quadratics-equation} at the point
  \((u,0)\). Note that this value is independent from \(\tau\). Since
  this is the tangent of the intersection angle with the real line,
  the square of the cosine of this angle is:
  \begin{displaymath}
    \frac{1}{1+(\frac{dv}{du})^2}=\frac{n^2}{l^2+k^2 u^2+n^2-2 k u
      l}=\frac{n^2}{l^2+n^2-km}=
    \frac{\scalar{Q}{\Space{R}{}}^2}{\scalar[e]{Q}{Q}},
  \end{displaymath}
  if \(ku^2-2  u l+m=0\).
\end{proof}
Also, we note that, the independence of the left-hand side
of~\eqref{eq:cosine-real-axis} from  \(x\) can be shown
from basic principles. Indeed, for a fixed \(t\) the subgroup
\(H_\tau\) acts transitively on the family of triples \(\{x,
\frac{x+\tau t}{tx+1}, \alli\}\), thus \(H_\tau\) acts transitively on
all bilinear forms orthogonal to such triples. However, the left-hand
side of~\eqref{eq:cosine-real-axis} is \(\SL\)-invariant, thus may not
depend on \(x\). This simple reasoning cannot provide the exact
expression in the right-hand side of~\eqref{eq:cosine-real-axis},
which is essential for the geometric interpretation of the Poincar\'e
extension.

To restore a cycle from its intersection points with the real line we
need also to know its cycle product with the real line. If this
product is non-zero then the sign of the parameter \(n\) is
additionally required.  At the cycles' language, a common point of
cycles \(\cycle{}{}\) and \(\cycle[\tilde]{}{}\) is encoded by a cycle
\(\hat{C}\) such that:
\begin{equation}
  \label{eq:cycles-passing-point}
  \scalar[e]{\cycle[\hat]{}{}}{\cycle{}{}}=
  \scalar[e]{\cycle[\hat]{}{}}{\cycle[\tilde]{}{}}=
  \scalar[\tau]{\cycle[\hat]{}{}}{\cycle[\hat]{}{}}=0.
\end{equation}
For a given value of \(\tau\), this produces two linear and one
quadratic equation for parameters of \(\hat{C}\). Thus, a pair of
cycles may not have a common point or have up to two such
points. Furthermore, M\"obius-invariance of the above
conditions~\eqref{eq:cosine-real-axis}
and~\eqref{eq:cycles-passing-point} supports the geometrical
construction of Poincar\'e extension, cf. Lem.~\ref{le:pencil-common-point}:
\begin{lem}
  \label{le:poincare-geom-general}
  Let a family consist of cycles, which are \(e\)-orthogonal to a
  given \(\tau\)-isotropic cycle \(\cycle[\hat]{}{}\) and have the
  fixed value of the fraction in the left-hand side
  of~\eqref{eq:cosine-real-axis}. Then, for a given M\"obius
  transformation \(g\) and any cycle \(\cycle{}{}\) from the
  family, \(g\cycle{}{}\) is \(e\)-orthogonal to the
  \(\tau\)-isotropic cycle \(g\cycle[\hat]{}{}\) and has the same
  fixed value of the fraction in the left-hand side
  of~\eqref{eq:cosine-real-axis} as \(\cycle{}{}\).
\end{lem}

Summarising the geometrical construction, the Poincar\'e extension based on
two intervals and the additional data produces two situations:
\begin{enumerate}
\item If the cycles \(C\) and \(\tilde{C}\) are orthogonal to the real line, then
  a pair of overlapping cycles produces a point of the elliptic upper
  half-plane, a pair of disjoint cycles defines a point of the
  hyperbolic. However, there is no orthogonal cycles uniquely defining
  a parabolic extension.
\item If we admit cycles, which are not orthogonal to the real line,
  then the same pair of cycles may define any of the three different
  types (EPH) of extension.
\end{enumerate}
These peculiarities make the extension based on three intervals,
described above, a bit more preferable.

\section{Concluding remarks}
\label{sec:concl-remarks-open}

Based on the consideration in
Sections~\ref{sec:cycles}--~\ref{sec:geom-cycl} we describe the
following steps to carry out the generalised extension procedure:
\begin{enumerate}
\item Points of the extended space are equivalence classes of aligned
  triples of cycles in \(P\Space{R}{1}\), see
  Defn.~\ref{de:aligned-triples-intervals}. The equivalence relation
  between triples will emerge at step~\ref{it:equivalence-triples}.
\item \label{it:one-param-subgroup} A triple \(T\) of different cycles
  defines the unique one-parameter continues subgroup \(S(t)\) of
  M\"obius transformations as defined in
  Prop.~\ref{pr:triple-subgroup}.
\item \label{it:equivalence-triples} Two triples of cycles are
  equivalent if and only if the subgroups defined in
  step~\ref{it:one-param-subgroup} coincide (up to a parametrisation).
\item The geometry of the extended space, defined by the equivalence
  class of a triple \(T\), is elliptic, parabolic or hyperbolic
  depending on the subgroup \(S(t)\) being similar \(S(t)=gH_\tau(t)
  g^{-1}\), \(g\in\SL\) (up to parametrisation) to
  \(H_\tau\)~\eqref{eq:a1n1k-universal} with \(\tau=-1\), \(0\) or
  \(1\) respectively. The value of \(\tau\) may be identified from the
  triple using Prop.~\ref{pr:type-of-subgroup}.
\item For the above \(\tau\) and \(g\in\SL\), the point of the
  extended space, defined by the the equivalence class of a triple
  \(T\), is represented by \(\tau\)-isotropic (see
  Defn.~\ref{de:product-cycles}(\ref{it:isotropic})) bilinear form
  \(g^{-1}
  \begin{pmatrix}
    1&\tau\\1&1
  \end{pmatrix}
  g\), which is \(S\)-invariant, see the end of
  Section~\ref{sec:triples-intervals}.
\end{enumerate}

Obviously, the above procedure is more complicated that the geometric
construction from Section~\ref{sec:geom-constr}. There are reasons for
this, as discussed in Section~\ref{sec:geom-cycl}: our procedure is
uniform and we are avoiding consideration of numerous subcases created
by an incompatible selection of parameters. Furthermore, our presentation
is aimed for generalisations to M\"obius transformations of moduli
over other rings. This can be considered as an analog of Cayley--Klein
geometries~\citelist{\cite[Apps.~A--B]{Yaglom79} \cite{Pimenov65a}}.

It shall be rather straightforward to adopt the extension for
\(\Space{R}{n}\). M\"obius transformations in \(\Space{R}{n}\) are
naturally expressed as linear-fractional transformations in Clifford
algebras~\cite{Cnops02a} with a similar classification of subgroups
based on fixed points~\cites{Ahlfors85a,Zoll87}. The M\"obius
invariant matrix presentation of cycles \(\Space{R}{n}\) is already
known~\citelist{\amscite{Cnops02a}*{(4.12)} \cite{FillmoreSpringer90a}
  \amscite{Kisil14a}*{\S~5}}. Of course, it is necessary to
enlarge the number of defining cycles from \(3\) to, say,
\(n+2\). This shall have a connection with Cauchy--Kovalevskaya
extension considered in Clifford analysis~\cites{Ryan90a,Sommen85a}.
Naturally, a consideration of other moduli and rings may require some
more serious adjustments in our scheme.

Our construction is based on the matrix presentations of cycles. This
techniques is effective in many different cases
\cites{Kisil12a,Kisil14a}. Thus, it is not surprising that such ideas
(with some technical variation) appeared independently on many
occasions~\citelist{\amscite{Cnops02a}*{(4.12)}
  \cite{FillmoreSpringer90a} \amscite{Schwerdtfeger79a}*{\S~1.1}
  \amscite{Kirillov06}*{\S~4.2}}. The interesting feature of the present
work is the complete absence of any (hyper)complex numbers. It deemed
to be unavoidable \amscite{Kisil12a}*{\S~3.3.4} to employ complex, dual
and double numbers to represent three different types of M\"obius
transformations extended from the real line to a plane. Also (hyper)complex
numbers were essential in~\cites{Kisil12a,Kisil06a} to define 
three possible types of cycle product~\eqref{eq:cycle-product-expl},
and now we managed without them.

Apart from having real entries, our matrices for cycles share the
structure of matrices from~\citelist{\amscite{Cnops02a}*{(4.12)}
  \cite{FillmoreSpringer90a} \cite{Kisil12a} \cite{Kisil06a}}. To
obtain another variant, one replaces the map
\(\map{i}\)~\eqref{eq:t-map-matrix} by
\begin{displaymath}
  \map{t}:\
  \begin{pmatrix}
    x_1&y_1\\
    x_2&y_2
  \end{pmatrix}
  \mapsto
  \begin{pmatrix}
    y_1&y_2\\
    x_1&x_2
  \end{pmatrix}\,.
\end{displaymath}
Then, we may define symmetric matrices in a manner similar
to~\eqref{eq:matrix-for-cycle-defn}: 
\begin{displaymath}
\cycle{t}{xy}=\frac{1}{2}M_{xy}\cdot\map{t}(M_{xy})=
  \begin{pmatrix}
    xy & \frac{x+y}{2}  \\
    \frac{x+y}{2} & 1
  \end{pmatrix}\,.
\end{displaymath}
This is the form of matrices for cycles similar
to~\citelist{\amscite{Schwerdtfeger79a}*{\S~1.1}
  \amscite{Kirillov06}*{\S~4.2}}. The property~\eqref{eq:covariance-Cxy}
with matrix similarity shall be replaced by the respective one with
matrix congruence: \(g\cdot \cycle{t}{xy} \cdot g^{t}=
\cycle{t}{x'y'}\). The rest of our construction may be adjusted
for these matrices accordingly.

\section*{Acknowledgments}
\label{sec:acknowledgments}
I am grateful to anonymous references for many useful suggestions.


\providecommand{\noopsort}[1]{} \providecommand{\printfirst}[2]{#1}
  \providecommand{\singleletter}[1]{#1} \providecommand{\switchargs}[2]{#2#1}
  \providecommand{\irm}{\textup{I}} \providecommand{\iirm}{\textup{II}}
  \providecommand{\vrm}{\textup{V}} \providecommand{\cprime}{'}
  \providecommand{\eprint}[2]{\texttt{#2}}
  \providecommand{\myeprint}[2]{\texttt{#2}}
  \providecommand{\arXiv}[1]{\myeprint{http://arXiv.org/abs/#1}{arXiv:#1}}
  \providecommand{\doi}[1]{\href{http://dx.doi.org/#1}{doi:
  #1}}\providecommand{\CPP}{\texttt{C++}}
  \providecommand{\NoWEB}{\texttt{noweb}}
  \providecommand{\MetaPost}{\texttt{Meta}\-\texttt{Post}}
  \providecommand{\GiNaC}{\textsf{GiNaC}}
  \providecommand{\pyGiNaC}{\textsf{pyGiNaC}}
  \providecommand{\Asymptote}{\texttt{Asymptote}}

\end{document}